\documentclass[a4paper,10pt]{article}
\usepackage{tikz} 

\usepackage[utf8]{inputenc}
\usepackage{amsmath,amsfonts,amssymb,amscd,amsthm,xspace, graphicx, xparse}
\usepackage[normalem]{ulem}
\usepackage[mathscr]{eucal}
\usepackage{color}
\usepackage{ulem}

\usepackage[colorlinks, linkcolor={black},citecolor={black}]
           {hyperref}

\usepackage{physics}

\newtheorem{theorem}{Theorem}
\newtheorem{lemma}[theorem]{Lemma}
\newtheorem{corollary}[theorem]{Corollary}

\theoremstyle{remark}
\newtheorem{remark}[theorem]{Remark}
\numberwithin{equation}{section}

\newtheorem{definition}[theorem]{Definition}

\NewDocumentEnvironment{manual}{O{theorem}m}
 {%
  \addtocounter{theorem}{-1}%
  \renewcommand\thetheorem{#2}%
  \begin{#1}
 }
 {\end{#1}}

\def\div{\hbox{\rm div}\,}

\def\loc{{\mathrm loc}}

\def\e{{\varepsilon}}

\def\XXint#1#2#3{{\setbox0=\hbox{$#1{#2#3}{\int}$ }
\vcenter{\hbox{$#2#3$ }}\kern-.6\wd0}}

\begin{document}

\title{Radiation of the energy-critical wave equation with compact support}

\author{Zhen Lei \footnotemark[1]\ \footnotemark[2]
\and Xiao Ren  \footnotemark[1]\ \footnotemark[3]
\and Zhaojie Yang \footnotemark[1]\ \footnotemark[4]}
\maketitle

\renewcommand{\thefootnote}{\fnsymbol{footnote}}
\footnotetext[1]{School of Mathematical Sciences; LMNS and Shanghai
Key Laboratory for Contemporary Applied Mathematics, Fudan University, Shanghai 200433, P. R.China.} \footnotetext[2]{Email: zlei@fudan.edu.cn}
\footnotetext[3]{Email: xiaoren18@fudan.edu.cn}
\footnotetext[4]{Email: yangzj20@fudan.edu.cn}

\begin{abstract}We prove exterior energy lower bounds for (nonradial) solutions to the energy-critical nonlinear wave equation in space dimensions $3 \le d \le 5$,  with compactly supported initial data. In particular, it is shown that nontrivial global solutions with compact spatial support must be radiative in the sense that at least one of the following is true:
\begin{itemize}
\item[(1)] $\int_{|x|> |t|} \left( |\partial_t u|^2 + |\nabla u|^2  \right) \mathrm{d}x \ge \eta_1(u) > 0, \ \mathrm{for} \ \mathrm{all} \ t \ge 0 \ \mathrm{or} \ \mathrm{all} \ t \le 0,$
\item[(2)] $\int_{|x|> -\varepsilon +|t|} \left( |\partial_t u|^2 + |\nabla u|^2  \right) \mathrm{d}x \ge \eta_2(\varepsilon, u) > 0, \ \mathrm{for} \ \mathrm{all} \ t \in \mathbb{R},  \varepsilon > 0.$
\end{itemize}
In space dimensions 3 and 4, a nontrivial soliton background is also considered. As an application, we obtain partial results on the rigidity conjecture concerning solutions with the compactness property, including a new proof for the global existence of such solutions.

\end{abstract}

\renewcommand{\thefootnote}{\arabic{footnote}}


\section{Introduction}
We consider the focusing energy-critical wave equation in $\mathbb{R}^{1+d}, 3 \le d \le 5$,
\begin{equation} \label{202}
\begin{cases}
\partial_t^2 u - \Delta u - |u|^{\frac{4}{d-2}} u = 0, \\
u|_{t = 0} = u_0 \in \dot{H}^1, \quad \partial_t u|_{t=0} = u_1 \in L^2. 
\end{cases}
\end{equation}
In the important work of Duyckaerts, Kenig and Merle \cite{DKM}, soliton resolution for radial solutions to \eqref{202} with uniformly bounded $\dot{H}^1\times L^2$ norm in $d=3$ is established.  The key step of the proof there is a characterization of the ground-state solution $W = \left( 1 + \frac{|x|^2}{d(d-2)} \right)^{1-\frac{d}{2}}$ as, up to rescaling and sign change, the only nonzero radial solution that does not satisfy the so-called channel of energy inequality. Among other results, it is shown in \cite{DKM} that, for any nonzero radial global solution $u$ to \eqref{202} in $d=3$, if $u$ is not identical to $\pm {\lambda^{\frac12}} W(\lambda x)$ for any $\lambda > 0$ and any sign $+$ or $-$, then there exists $A>0$ and $\eta > 0$ such that, for all $t \ge 0$ or for all $t \le 0$,
\begin{equation} \label{000}
\int_{|x|> A+|t|} \left(|\partial_t u|^2 + |\nabla u|^2\right)  \mathrm{d}x \ge \eta.
\end{equation}
Soliton resolution along a sequence of times for \eqref{202} without the radial symmetry assumption is proved in \cite{DJKM}. Recently, the radial soliton resolution theorem in all odd  space dimensions is established in \cite{DKModd}. Variations of the channel of energy estimate \eqref{000} have continued to play essential roles in these works.  

%

Do we have channel of energy estimates similar to \eqref{000} in the \emph{nonradial} setting?   As remarked in \cite{DJKM}, bounds of the same strength are not expected to be true, partly due to the fact that, even for the linear wave solutions, there is an infinite dimensional subspace of $\dot{H}^1 \times L^2$ for which \eqref{000} fails for $A>0$. In odd space dimensions, there is a lower bound for linear wave solutions in the exterior of the double light cone $|x| >  |t|$, \emph{i.e.},  with $A = 0$ in \eqref{000}, see \emph{e.g.}, \cite[Proposition 2.7]{DKM12JEMS}.  Note that it is difficult to extend this property to the nonlinear setting using perturbative arguments, unless there are smallness assumptions. In \cite{DKMCMP}, further results are proved for the linear wave equation with a potential obtained by linearizing \eqref{202} at the ground state $W$, still in odd space dimensions. For more results on channels of energy for linear waves, we refer to the recent work \cite{cote}.

The channel of energy estimates are closely related to the \emph{rigidity conjecture} for solutions satisfying the compactness property. Equation \eqref{202} admits infinitely many stationary solutions $Q \in \dot{H}^1(\mathbb{R}^d)$, that is, finite-energy solutions to
$$ - \Delta Q =  |Q|^{\frac{4}{d-2}} Q.$$
Taking the Lorentz transform of such a steady-state, one obtains travelling wave solutions (or solitons) to \eqref{202} for arbitrary $l \in \mathbb{R}^d$ with $|{l}| < 1$:
\begin{equation} \label{Qldefinition}
 Q_l(t,x) = Q\left(\mathbf{P}_2 x + \frac{1}{\sqrt{1-|l|^2}} \, \mathbf{P}_1 (x-lt)\right),
\end{equation}
with
\begin{equation}
  \mathbf{P}_1 = \frac{l\otimes l}{|l|^2}, \quad \mathbf{P}_2 = \mathbf{I} - \mathbf{P}_1.
\end{equation}
We say that a solution satisties the compactness property, if there exist functions $\lambda(t) >0$, $x(t) \in \mathbb{R}^d$, defined for $t \in (T^-(u), T^+(u))$, such that the set
\begin{align} 
K = &\Big\{\left(\lambda(t)^{\frac{d}{2}-1} u(t, \lambda(t) \cdot + x(t)), \lambda(t)^{\frac{d}{2}} u_t(t, \lambda(t) \cdot + x(t)\right), \nonumber \\
& \quad \quad \quad \quad \quad \quad \quad \quad \quad \quad \quad \quad \quad \quad \quad \quad t \in \left(T^-(u), T^+(u)\right)\Big\}
\end{align} 
is precompact in $\dot{H}^1 \times L^2$. Here $(T^-(u), T^+(u))$ is the maximal lifespan of $u$. The rigidity conjecture states that any solution to \eqref{202} with the compactness property is a travelling wave, see, \emph{e.g.}, \cite{DKMCPAA} and \cite{DKMCompact}.  It has been proved that if $u$ satisfies the compactness property, then $-T^-(u) = T^+(u) = +\infty$ (see Corollary \ref{thm4} and Remark \ref{rmk-cor}), and moreover, $u$ is \emph{nonradiative} in each time direction (see, \emph{e.g.}, \cite[Section 2]{DKMCPAA}): for any $A \in \mathbb{R}$,
\begin{equation}
\lim_{t \to \pm\infty} \int_{|x|>A+|t|} \left( |\partial_t u|^2 + |\nabla u|^2 \right) \, \mathrm{d}x = 0.
\end{equation}
In other words, for solutions with the compactness property, channel of energy estimates cannot hold. Clearly, travelling waves $Q_l$ satisfy the compactness property and are, in particular, nonradiative. In order to solve the rigidity conjecture, a natural strategy is to prove that any (global) solution to \eqref{202} which is not identical to a travelling wave is radiative, in the sense that \eqref{000} holds for some $A \in \mathbb{R}$.

We are now ready to present the main results. It turns out that \eqref{000} holds for the nonlinear problem \eqref{202} for an arbitrary $A<0$ if $u$ is a nontrivial global solution with  $(u_0, u_1)$ compactly supported. For one particular time direction only, \eqref{000} still holds for a sufficiently negative $A$. Moreover, the results are valid even with a soliton background in $d = 3$ and $d = 4$.  A direct consequence is that such solutions do not satisfy the compactness property. Note that no symmetry or size restrictions on the initial data are required here.

\begin{theorem}[Energy channel with compactly supported initial data] \label{thm2}
Under either of the assumptions
\begin{itemize}
\item[(i)] $3\le d \le 5$, $\emptyset \neq \mathrm{supp} (u_0, u_1) \subset \overline{B_R}$ for some $R > 0$.
\item[(ii)] $d = 3$ or $4$, $\emptyset \neq \mathrm{supp} (u_0 - Q_l(0, \cdot), u_1 - \partial_t Q_l(0, \cdot))  \subset \overline{B_R}$ for some $R > 0$ and a travelling wave solution $Q_l$.
\end{itemize}
the statements (A) and (B) on the solution $u$ to \eqref{202} are both true.
\begin{itemize}
\item[(A)] Suppose $- T^-(u) = T^+(u) = + \infty$, then at least one of the following holds:
\begin{itemize}
\item[(A1)] there exists $\eta_1(u) > 0$ such that for all $t\ge 0$ or all $t\le 0$,
$$\int_{|x|> |t|} \left( |\partial_t u|^2 + |\nabla u|^2  \right) \mathrm{d}x \ge \eta_1.$$
\item[(A2)] for any $\varepsilon > 0$, there exists $\eta_2(\varepsilon, u) > 0$ such that for all $t \in \mathbb{R}$,
$$\int_{|x|> -\varepsilon +|t|} \left( |\partial_t u|^2 + |\nabla u|^2  \right) \mathrm{d}x \ge \eta_2.$$
\end{itemize}
\item[(B)] Suppose $T^+(u) = + \infty$, then there exists $\eta_3(u) > 0$ such that for all $t \ge 0$,
$$\int_{|x|> - R +t} \left( |\partial_t u|^2 + |\nabla u|^2 \right) \mathrm{d}x \ge \eta_3.$$
\end{itemize}
\end{theorem}
%
%

\begin{remark}
Unlike the existing results in the radial setting \cite{DKM} or in the linear setting \cite[Proposition 2.7]{DKM12JEMS}, the numbers $\eta_{1,2,3}$ in Theorem \ref{thm2} may be arbitrarily small with respect to the energy of the initial data. 
\end{remark}

\begin{remark}
In our proof, the main obstruction for extending the result to $A \ge 0$ is H\"ormander's geometric pseudoconvexity condition in unique continuation theory. Under the assumption (i), if for some $A \ge 0$,
$$\lim_{t \to + \infty}\int_{|x|> A + t} \left( |\partial_t u|^2 + |\nabla u|^2  \right) \mathrm{d}x = 0,$$ 
then we can prove the vanishing of $u$ in
$$\{ (t,x): |x|^2 - (t+A)^2 > R^2 - A^2, \ t \ge 0 \}.$$
On the other hand, for linear waves, using Holmgren's unique continuation theorem (see \cite[Theorem 1.6]{LBook}) instead of H\"ormander's, we can obtain a larger vanishing region
$$\left\{ (t,x): |x|  >  \left| t-\frac{R-A}{2} \right| + \frac{R+A}{2}, \ t \ge 0 \right\}.$$
\end{remark}

\begin{remark}
We do not assume the uniform boundedness of the $\dot{H}^1 \times L^2$ norm for $u$ in the above results.
\end{remark}

Next, we give an application of the above channel of energy properties to the rigidity conjecture mentioned earlier.

\begin{corollary} \label{thm4} 
For $3 \le d \le 5$, let $u$ be a solution to \eqref{202} with the compactness property. Then $(T^-, T^+) = (-\infty, +\infty)$. Moreover, $u$ cannot satisfy either of the assumptions (i) and (ii) from Theorem \ref{thm2}.
\end{corollary}

\begin{remark} \label{rmk-cor}
This global existence statement was first obtained in \cite{DKMHP}. We provide a different proof for this result, based on Theorem \ref{thm2}. Solutions to \eqref{202} under either of the assumptions (i) and (ii) from Theorem \ref{thm2} cannot satisfy the compactness property, and moreover, if they are global in one time direction with uniformly bounded $\dot{H}^1 \times L^2$ norm, the scattering profile constructed in \cite{jems16} must be nontrivial.
\end{remark}

Our main tools are the conformal inversion for the Minkowski spacetime which brings the null infinities to the light cone $\{|y| = |s|\}$ (see Section \ref{sec3}), and the unique continuation theory established in \cite{KT} and \cite{Dos}. Unique continuation from infinity results have appeared in the works \cite{Shao1, Shao2}. The method in \cite{Shao2} establishes uniqueness results for the linear operator $\Box + V$ assuming global pointwise bounds on  smooth solutions and the potential $V$ in the exterior region $|x| > |t|$. The focusing superconformal equation \eqref{202} is not covered by the analysis there, see \cite[Section 1.1.2]{Shao2}. In \cite{Shao1}, curved spacetime is treated and vanishing to infinite order at the null infinities are required. Our new observation here is that, the vanishing of energy in the null infinity and the critical regularity in the current setting are just sufficient for the application of continuation from infinity arguments. In particular, we reduce the problem to an ill-posed characteristic problem studied in, \emph{e.g.}, \cite[Theorem 1.1]{IK} and \cite{Ler}, where vanishing data are given on two transversal characteristic surfaces for the wave operator $\Box = \partial_t^2 - \Delta_x$. 

In Section \ref{sec2}, we recall some known facts on Carleman estimates and unique continuation, as well as localized well-posedness results for nonlinear wave equations with variable coefficients. In Section \ref{sec3}, we apply the conformal inversion transform to \eqref{202}. In Section \ref{sec4}, we prove regularity and vanishing properties for a transformed solution assuming the decay of energy  for the original solution in the null directions. The proofs will be concluded in Section \ref{sec5}.

\section{Preliminaries} \label{sec2}

\subsection{Carleman estimates and unique continuation}

One of the key tools in our analysis is the unique continuation theory for the wave operator with a critical potential. To recall the key ideas, we consider an oriented hypersurface in $\mathbb{R}^n$ defined by a smooth function $\phi$, $\Gamma = \{ \phi = 0\}$. Denote the two sides of $\Gamma$ by $\Gamma^+ = \{\phi>0\}$ and $\Gamma^- = \{\phi < 0\}$. 

\begin{definition}
We say that the unique continuation property across the hypersurface $\Gamma$ holds for a differential operator $P(x,D)$ if for each $x_0 \in \Gamma$ there exists a neighbourhood $V$ of $x_0$ such that: if $u$ is a solution to $P(x,D)u = 0$ in $V$ with $u = 0$ in $\Gamma^+ \cap V$, then $u = 0$ near $x_0$.
\end{definition}

Unique continuation for the wave operator requires a geometric \emph{pseudoconvexity} condition on the hypersurface $\Gamma$ (or equivalently on $\phi$).The full definition of the pseudoconvexity condition for general operators can be found in \cite[Definition 8.2]{KT} or \cite[Definition 4.7]{LBook}. For the operator $\Box$, the definition is simpler if restricted to the noncharacteristic case. 

\begin{definition} \label{def-pc}
We say that the oriented hypersurface $\Gamma$ (with $\Gamma^+$ understood as above), or equivalently the defining function $\phi$, is noncharacteristic and strongly pseudoconvex at point $(t,x) \in \Gamma$ with respect to $\Box$, if both statements below are satisfied in a $\Gamma$-neighbourhood of $(t,x)$.
\begin{enumerate}
\item $- (\partial_t \phi)^2 + |\nabla_x \phi|^2 \neq 0,$
\item For any $ (p, \xi) \in TV$ with $\xi$ null, $ \partial_\xi^2 \phi|_{p} = \xi^i\xi^j\partial_i\partial_j \phi|_{p}> 0$.
\end{enumerate}
\end{definition}

Note that the standard timelike hyperboloids $\Gamma$ defined by 
$$\phi_0 = |x|^2 - t^2 - a^2 = 0$$
with either orientations rest on the borderline of pseudoconvexity. For our purpose, three kinds of strongly pseudoconvex surfaces will be used. For each $\phi_k$ below, $\Gamma = \{\phi_k = 0\}$ with $\Gamma^+ = \{\phi_k > 0\}$ satisfies Definition \ref{def-pc} at every point $p \in \Gamma$.
\begin{enumerate}
\item  $\phi_1 = (|x|+\delta)^2 - t^2 - a^2$, $ 0 < \delta < a$.
\item  $\phi_2 = -(|x|-\delta)^2 + t^2 + a^2$, $\delta > 0$, in $|x| > \delta$.
\item  $\phi_3 = |x| - a$, $a>0$.
\end{enumerate}

Next we state an important unique continuation theorem, which generalizes the classical results of H\"ormander \cite{Hormander}, and Kenig, Ruiz and Sogge \cite{KRS}.  

\begin{lemma} \label{lem8}
Let $\phi$ be a noncharacteristic and strongly pseudoconvex function with respect to $\Box$ in $\mathbb{R}^{d+1}$. Then there exists $\tau_0 > 0$ such that, for compactly supported $u$ and $\tau > \tau_0$, we have
\begin{equation} \label{carleman}
\|e^{\tau \phi} u\|_{L_{t,x}^{\frac{2(d+1)}{d-1}}} \le C_d \|e^{\tau \phi} \Box u\|_{L_{t,x}^{\frac{2(d+1)}{d+3}}},
\end{equation} 
where $C_d$ is independent of $\tau$. As a consequence, unique continuation for (sufficiently regular) solutions $u$ to $(\Box + V) u = 0$ across smooth noncharacteristic and strongly pseudoconvex hypersurfaces holds for potentials $V \in L_{t,x}^{\frac{d+1}{2}}$.
\end{lemma} 
The proof of this theorem is quite involved, see \cite[Theorems 8.14 and 8.15]{KT} or \cite{Dos}. The minimal regularity requirement on $u$ is $\Box u \in L^\frac{2(d+1)}{d+3}_{\loc}$, $u \in L^{\frac{2(d+1)}{d-1}}_\loc$, and  $\partial_{t,x} u \in L^\frac{2(d+1)}{d+3}_{\loc}$. The last condition shows up because, in the proof of unique continuation,  \eqref{carleman} must be applied to $\chi u$ where $\chi$ is a suitable smooth cutoff function. Hence, one needs
\begin{equation}
\Box(\chi u) = \chi \Box u + 2 \partial_t \chi \partial_t u -2 \nabla_x \chi \cdot \nabla_x u + u \Box \chi \in L^\frac{2(d+1)}{d+3}.
\end{equation}
See, \emph{e.g.}, \cite[Theorem 9.24]{LBook} for the standard arguments.

\subsection{Localized well-posedness for NLW with variable coefficients}

It is well-known that the NLW equation \eqref{202} is locally well-posed using Strichartz norms $\|u\|_{L_{t,x}^{\frac{2(d+1)}{d-2}}} + \|D^\frac12 u\|_{L_{t,x}^{\frac{2(d+1)}{d-1}}}$ or $\|u\|_{L_{t}^{\frac{d+2}{d-2}}L_x^{\frac{2(d+2)}{d-2}}}$, see, \emph{e.g.}, \cite{K}. Using the universal extension operator for Sobolev spaces (see \cite[Section 7.69]{AF}) and the finite speed of propagation for linear wave equation, this result can be localized in a truncated backward light cone $\mathcal{K}_{T,T_0} : = \{(t,x): 0\le t\le T_0, |x| \le T-t\}$ with $T_0<T$. 

	\begin{lemma} \label{wp} Consider the following NLW equations:
		\begin{itemize}
			\item $3 \le d \le 5$, $\partial_t^2 u - \Delta u = F(t,x) \, |u|^{\frac{4}{d-2}} u$.
			
			\item $d=3$, $\partial_t^2 u - \Delta u = \sum_{j = 1}^5 F^{(3)}_j(t,x) u^j$.
			
			\item $d=4$, $\partial_t^2 u - \Delta u = \sum_{j = 1}^3 F^{(4)}_j(t,x) u^j$.
%
			
		\end{itemize}
with initial data in the ball,
$$u|_{t = 0} = u_0 \in H^1(B_T(0)), \quad \partial_t u|_{t=0} = u_1 \in L^2(B_T(0)).$$
We assume that all coefficients involved, \textit{i.e.}, $F, F_j^{(3)}, F_j^{(4)}$ are uniformly bounded and continuous in spacetime. Then we have the following well-posedness results localized in $\mathcal{K}_{T,T_0}$:
\begin{itemize}
\item (Local well-posedness) There exists a unique maximal solution in $\mathcal{K}_{T,T_0}$ satisfying $(u, \partial_t u) \in C([0,T^+), H^1\times L^2(B_{T-t}))$, $T^+ \in (0, T_0]$, and for any $T' < T^+$,
$$u\in L_{t}^{\frac{d+2}{d-2}}L_x^{\frac{2(d+2)}{d-2}}(\mathcal{K}_{T,T'}).$$ 
Interpolation also gives $u \in L_{t,x}^{\frac{2(d+1)}{d-2}}(\mathcal{K}_{T,T'})$.

\item (Small-data well-posedness) There exists a constant $\delta$ depending on $T$, $T_0$ and the $C^0$-norm of the coefficients such that if $\|(u_0, u_1)\|_{H^1 \times L^2(B_T)} < \delta$, then in the above result, $T^+ = T_0$.

\item (Blow-up/regularity criterion) Suppose that the coefficients $F, F_j^{(3)}, F_j^{(4)}$ are $\frac12$-H\"older continuous in $\mathbb{R}^d$ uniformly in time. If the solution blows up before $T_0$, i.e., $T^+ < T_0$, then
$$\|u\|_{L_{t,x}^{\frac{2(d+1)}{d-2}}(K_{T,T^+})} = + \infty.$$
\end{itemize}

\end{lemma}

The results are formulated using inhomogeneous Sobolev norms for the sake of Sobolev extension. The proof is based on the following localized Strichartz estimates for linear wave equation $\Box u = \mathcal{F}$ in $K_{T,T_0}$ which follows easily from the corresponding estimates on whole space, Sobolev extension and finite speed of propagation,
\begin{align}
& \quad \|(u, \partial_t u)\|_{C_t(H^1 \times L^2)(K_{T,T_0})}  + \|u\|_{L_{t}^{\frac{d+2}{d-2}}L_x^{\frac{2(d+2)}{d-2}}(K_{T,T_0})} \nonumber \\
&  \lesssim \|(u_0, u_1)\|_{H^1 \times L^2(B_T)} + \|\mathcal{F}\|_{L_t^1L_x^2(K_{T,T_0})},
\end{align}
and
\begin{align}
& \quad \|(u, \partial_t u)\|_{C_t(H^1 \times L^2)(K_{T,T_0})}  + \|u\|_{L_{t,x}^{\frac{2(d+1)}{d-2}}(K_{T,T_0})} \nonumber \\
& + \|u\|_{L_t^{\frac{2(d+1)}{d-1}}W_x^{\frac12,\frac{2(d+1)}{d-1}}(K_{T,T_0})} + \|u\|_{L_{t}^{\frac{d+2}{d-2}}L_x^{\frac{2(d+2)}{d-2}}(K_{T,T_0})} \nonumber \\
&  \lesssim \|(u_0, u_1)\|_{H^1 \times L^2(B_T)} + \|\mathcal{F}\|_{L_t^{\frac{2(d+1)}{d+3}}W_x^{\frac12,\frac{2(d+1)}{d+3}}(K_{T,T_0})}.
\end{align}
The rest of the proof follows the same arguments for the Cauchy problem, see \cite[Chapter 1]{K} for details.

 It is interesting to apply the more sophisticated methods in  \cite{Bulut} to extend our results to $d = 6$, which we do not pursue here. In the proof of Theorem \ref{thm2}, the main obstruction for extending to high dimensions $d>7$ is the low regularity of the coefficients in a transformed equation, see \eqref{30300} below.

\section{Conformal inversion} \label{sec3}

In order to obtain unique continuation from infinity results, we have to use the following coordinate change which is conformal with respect to the Minkowski metric. Under the notations of Theorem \ref{thm2}, define $w = u - Q_l$, where  $Q_l$ is set as $0$ if $u$ satisfies the assumption (i) from Theorem \ref{thm2}. By assumption, $\mathrm{supp} (w(0, \cdot), \partial_t w(0, \cdot)) \subset \overline{B_R}$ for some $R>0$. Since $Q_l$ is a solution to \eqref{202}, using the principle of finite propagation speed, $\mathrm{supp}\, w \subset \overline{B_{R+|t|}}$ for any $t \in (T^-(u), T^+(u))$. Suppose that $T^+(u) = + \infty$ and we shall work in the positive times direction $t \ge 0$. 

Using the coordinate change
\begin{equation} \label{3001}
t = \frac{s}{|y|^2 - s^2}, \quad  x = \frac{y}{|y|^2 - s^2},
\end{equation}
we define the transformed solution
\begin{equation} \label{3002}
v(s, y) = (|y|^2 -s^2)^{-\frac{d-1}{2}} w\left(\frac{s}{|y|^2 - s^2}, \frac{y}{|y|^2 - s^2}\right)
\end{equation}
in $|y| > |s|$. (This definition of $v$ in the region $|y| < |s|$ will not be used.) To see that the coordinate change $(t,x) \mapsto (s, y)$ is a conformal transformation, we set
\begin{equation}
\begin{cases}
\rho = |x|+t, \sigma = |x|-t, \\
\bar{\rho} = |y|+s, \bar{\sigma} = |y|-s.
\end{cases}
\end{equation}
It follows that, with $r = |x|, \bar{r} = |y|$,
\begin{equation}
\begin{cases}
\partial_\rho = \frac12(\partial_r + \partial_t), \ \partial_{\sigma} = \frac12(\partial_r - \partial_t), \\
\partial_{\bar{\rho}} = \frac12(\partial_{\bar{r}} + \partial_s), \ \partial_{\sigma} = \frac12(\partial_{\bar{r}} - \partial_s).
\end{cases}
\end{equation}
In terms of $\rho, \sigma$ variables, the above coordinate change is equivalent to $\bar{\rho} = \sigma^{-1}, \bar{\sigma} = \rho^{-1}$, and one can check that $\mathrm{d}t^2 - \mathrm{d}x^2 = \rho^2 \sigma^2 (\mathrm{d}s^2 - \mathrm{d}y^2)$. The null infinity in $(t,x)$ coordinates is mapped to the null cone $|y| = |s|$.

\begin{figure}[!h]
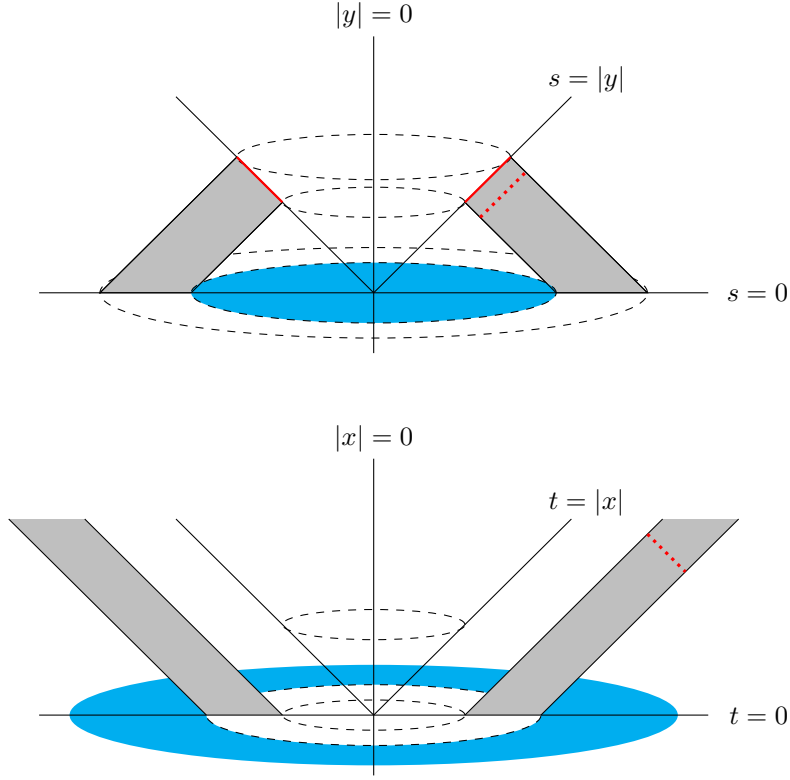

\centering
\tikz[scale = 2] {  

\fill [color = cyan]  (0,0-2.8) ellipse (2 and 0.3333);
\fill [color = white]   (0,0-2.8) ellipse (1.1 and 1.1*0.1833);

\draw [dashed] (0,0-2.8) ellipse (0.6 and 0.1);
\draw [dashed] (0,0-2.8) ellipse (1.1 and 1.1*0.1833);

\fill [color = lightgray] (0.6,0-2.8) -- (1.1, 0-2.8) -- (2.4, 1.3-2.8) -- (1.9, 1.3-2.8) -- (0.6,0-2.8);
\fill [color = lightgray] (-0.6,0-2.8) -- (-1.1, 0-2.8) -- (-2.4, 1.3-2.8) -- (-1.9, 1.3-2.8) -- (-0.6,0-2.8);

\draw (-2.2,0-2.8) -- (2.2,0-2.8);
\draw (0, -0.4-2.8) -- (0,1.7-2.8);
\draw (0, 0-2.8) -- (1.3, 1.3-2.8);
\draw (0, 0-2.8) -- (-1.3, 1.3-2.8);

\draw (0.6, 0-2.8) -- (1.9 , 1.3-2.8);
\draw (-0.6, 0-2.8) -- (-1.9, 1.3-2.8);

\draw (1.1, 0-2.8) -- (2.4 , 1.3-2.8);
\draw (-1.1, 0-2.8) -- (-2.4, 1.3-2.8);

\draw [dotted, color=red, very thick] (1.8,1.2-2.8)  -- (2.05,0.95-2.8);

\draw [dashed] (0,0.6-2.8) ellipse (0.6 and 0.1);


\node[label=above:{$t=|x|$}]  at (1.4,1.2-2.8) {};
\node[label=left:{$t = 0$}]  at (2.85, 0-2.8) {};
\node[label=right:{}]  at (-2.85, 0-2.8) {};
\node[label=below:{$|x| = 0$}]  at (0, 2.05-2.8) {};

\draw [dashed] (0,0.6) ellipse (0.6 and 0.1);

\draw [dashed] (0,0) ellipse (1.8 and 0.3);
\draw [fill = lightgray] (1.2,0) -- (0.6, 0.6) -- (0.9, 0.9) -- (1.8, 0) -- (1.2,0);
\draw [fill = lightgray] (-1.2,0) -- (-0.6, 0.6) -- (-0.9, 0.9) -- (-1.8, 0) -- (-1.2,0);

\draw [dashed, fill = cyan] (0,0) ellipse (1.2 and 0.2);

\draw (-2.2,0) -- (2.2,0);
\draw (0, -0.4) -- (0,1.7);
\draw (0, 0) -- (1.3, 1.3);
\draw (0, 0) -- (-1.3, 1.3);

\draw (1.2, 0) -- (0.6 , 0.6);
\draw (-1.2, 0) -- (-0.6, 0.6);

\draw (1.8, 0) -- (0.9 , 0.9);
\draw (-1.8, 0) -- (-0.9, 0.9);
\draw [dashed] (0,0.9) ellipse (0.9 and 0.15);

\draw [dotted, color=red, very thick] (0.7,0.5)  -- (1, 0.8);

\draw[red, line width=0.3mm] (0.6,0.6) -- (0.9,0.9);
\draw[red, line width=0.3mm] (-0.6, 0.6) -- (-0.9, 0.9);

\node[label=above:{$s = |y|$}]  at (1.4,1.2) {};
\node[label=left:{$s = 0$}]  at (2.85, 0) {};
\node[label=right:{}]  at (-2.85, 0) {};
\node[label=below:{$|y| = 0$}]  at (0, 2.05) {};
}

\caption{ Spacetime diagrams in $(s, y)$ and $(t, x)$ coordinates. The gray regions represent the null channel $\{R^{-1} - s \le |y| \le A^{-1} - s, |y| > s \ge 0\} \leftrightarrow \{A + t \le |x| \le R + t, t \ge 0\}$. The red solid lines in $(s,y)$ coordinates correspond to the null infinities along the channel. The initial data of $v$ and $w$ vanish in the corresponding blue domains.}
\end{figure}

To derive the equation for $v$, we compute
\begin{align*}
\square_{t,x} &= (\partial_t + \partial_r)(\partial_t - \partial_r) - \frac{d-1}{r}\partial_r - \frac{1}{r^2} \Delta_\omega \\
&= -4 \partial_\rho \partial_\sigma - \frac{2(d-1)}{\rho + \sigma} (\partial_\rho + \partial_\sigma)  - \frac{4}{(\rho + \sigma)^2} \Delta_\omega \\
&=-4\bar{\rho}^2 \bar{\sigma}^2 \partial_{\bar\sigma} \partial_{\bar\rho} + \frac{2(d-1) \bar\rho \bar\sigma}{\bar\rho + \bar\sigma} \left(\bar\sigma^2 \partial_{\bar\sigma} + \bar{\rho}^2 \partial_{\bar\rho}\right) - \frac{4\bar\rho^2\bar\sigma^2}{(\bar\rho+\bar\sigma)^2} \Delta_\omega \\
&=\bar{\rho}^2 \bar{\sigma}^2 \square_{s,y} + 2(d-1) \bar{\rho} \bar{\sigma} (\bar\rho \partial_{\bar\rho} + \bar\sigma \partial_{\bar\sigma}),
\end{align*}
where $\Delta_\omega$ is the Laplace operator in the angular variables. Also notice that
\begin{align*}
\square_{s,y} (\bar\rho \bar\sigma)^{\frac{d-1}{2}} = -2 (d-1)^2 (\bar\rho \bar\sigma)^{\frac{d-3}{2}}
\end{align*}
Hence,
\begin{align} \label{3005}
\square_{t,x} w  & = \square_{t,x} \left( (\bar\rho \bar\sigma)^{\frac{d-1}{2}} v \right) \nonumber \\
&= \bar{\rho}^\frac{d+3}{2} \bar{\sigma}^\frac{d+3}{2} \square_{s,y} v  -2 (d-1)^2 (\bar\rho \bar\sigma)^{\frac{d+1}{2}} v -2(d-1) (\bar{\rho}\bar{\sigma})^\frac{d+1}{2} (\bar\rho \partial_{\bar\rho} + \bar\sigma \partial_{\bar\sigma})  v  \nonumber \\
&\quad  + 2(d-1) \bar{\rho} \bar{\sigma} (\bar\rho \partial_{\bar\rho} + \bar\sigma \partial_{\bar\sigma}) \left( (\bar\rho \bar\sigma)^{\frac{d-1}{2}} v \right) \nonumber \\
&= \bar{\rho}^\frac{d+3}{2} \bar{\sigma}^\frac{d+3}{2} \square_{s,y} v.
\end{align}
Now consider the following three cases:


\smallskip

{\sc Case I}: $Q_l = 0$. In this case, $w$ is identical to $u$. Using \eqref{3005} and \eqref{202}, $v$ satisfies a wave equation with variable coefficients in the exterior region $|y| > |s|$,
\begin{equation} \label{30300}
\square_{s,y} v = (|y|^2 -s^2)^\frac{2}{d-2} |v|^{\frac{4}{d-2}} v.
\end{equation}
We consider the solution $\tilde{v}$  to the Cauchy problem
\begin{equation} \label{transeq}
\begin{cases}
\square_{s,y} \tilde{v} = (|y|^2 - s^2)_+^\frac{2}{d-2} |\tilde{v}|^{\frac{4}{d-2}} \tilde{v}, \\
\tilde{v}|_{s = 0} = v_0, \quad \partial_s \tilde{v}|_{s=0} = v_1,
\end{cases}
\end{equation}
with initial data given by $v$,
\begin{equation} \label{3008}
v_0(y) = |y|^{-d+1} u(0,\frac{y}{|y|^2}), \quad v_1(y) = |y|^{-d-1} \partial_t u(0,\frac{y}{|y|^2}).
\end{equation}
Here, $(|y|^2 - s^2)_+ = \mathbf{1}_{|y| > |s|} \cdot (|y|^2 - s^2)$. The transformed data $(v_0, v_1)$ vanish in $B_{R^{-1}}$ thanks to the compact support assumption of $w$, and moreover,
\begin{equation} \label{3009}
\|v_0\|_{\dot{H}_y^1} \le C_d R \|u(0, \cdot)\|_{\dot{H}_x^1}, \quad \|v_1\|_{L_y^2} \le C_d R \|\partial_t u(0, \cdot)\|_{L_x^2}.
\end{equation}
By Lemma \ref{wp}, problem \eqref{transeq} is well-posed locally in spacetime --- although the coefficient $(|y|^2 - s^2)_+^\frac{2}{d-2}$ is not uniformly bounded, the local existence holds in finite backward cones. Using the support of $(v_0, v_1)$ and finite speed of propagation, we have
\begin{equation} \label{3012}
\tilde{v} = 0, \quad \mathrm{in} \ \ |y| + |s| \le R^{-1}.
\end{equation}
Since the coordinate change is locally smooth for $|y| > |s|$ and 
\begin{equation}
u \in L_{t,x}^{\frac{2(d+1)}{d-2}}([0,T] \times \mathbb{R}^d),
\end{equation}
for any $T>0$, we deduce that
\begin{equation}  \label{3011}
v \in L_{s,y}^{\frac{2(d+1)}{d-2}}(|y| - \delta > s \ge 0, |y| + s \le \delta^{-1}),
\end{equation}
for any $\delta > 0$. By the finite speed of propagation and Lemma \ref{wp}\footnote{To be precise, we need to use a smooth approximation argument here. We mollify the initial data of $w$ to get approximate solutions $w_k$. Then the transformed solutions $v_k$ have smooth initial data approximating those of $v$. Local energy and Strichartz estimates for $v_k$ outside the light cone can be controlled uniformly due to the regularity criterion in Lemma \ref{wp}.}, we have
\begin{equation} \label{vv}
\tilde{v} = v, \quad \mathrm{in} \  |y| > s \ge 0,
\end{equation}
and  they are locally regular in $|y| > s \ge 0$ in the sense that for any $\delta > 0$,
\begin{align} \label{v-reg-ext} 
(v, \partial_t v) &\in C_s(\dot{H}_y^1 \times L_y^2)\left(|y| - \delta > s \ge 0, |y| + s \le \delta^{-1}\right),\\
\label{v-reg-ext2} v &\in L_s^{\frac{d+2}{d-2}}L_y^{\frac{2(d+2)}{d-2}}\left(|y| - \delta > s \ge 0, |y| + s \le \delta^{-1}\right).
\end{align}
\emph{In the sequel we do not distinguish between $\tilde{v}$ and $v$.} Due to the singularity in the coordinate change, we can not set $\delta = 0$ in the above results directly. In other words, the regularity of $v$ at the forward light cone with $|y| = s > (2R)^{-1}$ is not clear at this moment. Even though the coefficient in the nonlinearity degenerates at $|y| = s$, it still allows Type I (ODE) blow-up at first glance. This issue will be treated in the next section. 

\begin{remark}[Lifespan of $v$] \label{rmk-timeexistence}
Due to \eqref{3012} and \eqref{v-reg-ext}--\eqref{v-reg-ext2}, $v$ is locally regular in $[0,(2R)^{-1})_t \times \mathbb{R}^d$. Such a time of existence can be improved slightly as follows. Using the smallness of $\|(v_0, v_1)\|_{\dot{H}^1 \times L^2}$ in the domain $B_{R^{-1}+\delta}$ for small $\delta>0$, and the small-data well-posedness result in Lemma \ref{wp}, we know that $v$ is regular in the cone $|y| + |s| < R^{-1} + \delta$. Combined with \eqref{v-reg-ext}--\eqref{v-reg-ext2}, we see that $v$ is locally regular in $[0,(2R)^{-1}+{\delta}/{2})_t \times \mathbb{R}^d$. We will further improve this time interval in Section \ref{sec4}, assuming decay of energy in the null channels. This remark applies to the next two cases as well.
\end{remark}


\smallskip

{\sc Case II}: $Q_l \neq 0$, $d = 3$. By \eqref{202}, $w = u - Q_l$  satisfies the equation
\begin{equation} \label{00302}
\Box_{t,x} w = 5 Q_l^4 w + 10 Q_l^3 w^2 + 10 Q_l^2 w^3 + 5 Q_l w^4 + w^5.
\end{equation}
Let
\begin{equation}
\overline{Q_l}(s,y) = Q_l\left(\frac{s}{|y|^2 - s^2}, \frac{y}{|y|^2 - s^2}\right),
\end{equation} 
and, for simplicity, write $I = |y|^2 - s^2$. By \eqref{3005}, the equation
\begin{align} \label{315}
\Box_{s,y} v  &=  5 I^{-2} \overline{Q_l}^4 v + 10 I^{-1} \overline{Q_l}^3 v^2 + 10 \overline{Q_l}^2 v^3 + 5 I \overline{Q_l} v^4 + I^2 v^5 \nonumber  \\
&=:\sum_{i=1}^5 G_j^{(3)} v^j.
\end{align}
holds in $|y| > |s|$. Extend $G_j^{(3)}$ by zero into $|y|<s$. Similar to Case I, we consider the solution $\tilde{v}$ to the Cauchy problem
\begin{equation} \label{transeqd3}
\begin{cases}
\square_{s,y} \tilde{v} = \sum_{j=1}^5 G_j^{(3)} \tilde{v}^j, \\
\tilde{v}|_{s = 0} = v_0, \quad \partial_s \tilde{v}|_{s=0} = v_1,
\end{cases}
\end{equation}
with initial data $(v_0, v_1)$ defined as in \eqref{3008} with $u$ replaced by $w$. \eqref{3009}--\eqref{v-reg-ext2} with $u$ replaced by $w$ are still valid for similar reasons. Again, identify $\tilde{v}$ and $v$, so that $v$ may be extended into the region $|y| \le s$. 

Next, we prove the regularity of $G_j^{(3)}$ at the light cone $|y| = s$, so that local well-posedness theory can be applied to \eqref{transeqd3}. Let us recall the smoothness and decay properties of the steady-states
$Q$ which are summarized in \cite[Section 3]{DKMCompact}. For both $d = 3$ and $d = 4$, $Q \in C^\infty(\mathbb{R}^d)$, and for any $\alpha \in \mathbb{N}^d$,
\begin{equation} \label{Qdecay}
|\partial_x^\alpha Q(x)| \lesssim_{\alpha, Q} |x|^{-d+2-|\alpha|}, \ |x| \ge 1.
\end{equation}
Denote
\begin{equation} \label{KQdef}
K(Q) = \sup_{x \in \mathbb{R}^d} \left\{ |x|^{d-2} |Q(x)| +  |x|^{d-1}|\nabla Q(x)| \right\} < + \infty.
\end{equation}
Without loss of generality, consider $l = (l_1, 0, 0, 0)$ so that
\begin{equation}
Q_l(t,x) = Q\left(\frac{x_1 - l_1 t}{\sqrt{1-l_1^2}}, x_2, x_3\right).
\end{equation}
Then, in $|y| > s \ge 0$,
\begin{align} \label{320}
\left| \overline{Q_l}(s,y) \right| &=  \left| Q\left(I^{-1} \frac{y_1 - l_1 s}{\sqrt{1-l_1^2}}, I^{-1} y_2, I^{-1} y_3\right) \right| \nonumber \\
&\le K(Q) I \left( \left( \frac{y_1 - l_1 s}{\sqrt{1-l_1^2}}  \right)^2 + y_2^2 + y_3^2 \right)^{-\frac12} \nonumber \\
&\lesssim K(Q) I |y|^{-1}.
\end{align}
For the last line, we used the inequality
\begin{equation}
\left( \frac{y_1 - l_1 s}{\sqrt{1-l_1^2}}  \right)^2 + y_2^2 + y_3^2 \gtrsim |y|^2, \quad \mathrm{for} \ |y| > s \ge 0,
\end{equation}
with constant depending only on $l_1$, which can be proved by considering two separate cases $|y_1| > \frac{1+l_1}{2} |y|$ and $|y_1| \le \frac{1+l_1}{2} |y|$. Similarly, we have, in $|y| > s \ge 0$,
\begin{align} \label{321}
\left| \nabla_y \overline{Q_l}(s,y) \right| &\lesssim  \left(I^{-1} + I^{-2} |y|^2\right)\left| (\nabla_x Q)\left(I^{-1} \frac{y_1 - l_1 s}{\sqrt{1-l_1^2}}, I^{-1} y_2, I^{-1} y_3\right) \right| \nonumber \\
&\lesssim K(Q) \left(I^{-1}+I^{-2}|y|^{2}\right) I^2 |y|^{-2} \nonumber \\
&\lesssim K(Q).
\end{align}
From \eqref{315}, \eqref{320} and \eqref{321}, we deduce, in $|y| > s \ge 0$,
\begin{align} \label{G3bound}
\left|G_j^{(3)}(s,y)\right| &\lesssim K(Q)^{5-j} I^2 |y|^{j-5} \le K(Q)^{5-j} |y|^{j-1}, \quad 1 \le j \le 5,
\end{align}
and
\begin{align} \label{dG3bound}
\left|\nabla_y G_j^{(3)}(s,y)\right| &\lesssim K(Q)^{5-j} |y|^{j-2}, \quad 1 \le j \le 5.
\end{align}
In particular, \eqref{G3bound} shows that $G_j^{(3)}$ is continuous at $|y| = s > 0$ after extending by $0$ into $|y| < s$. Note that the singularity for $\nabla_y G_j^{(3)}$ is irrelavant since $v$ is supported away from the spacetime origin. The method here will also be useful in the proof of Lemma \ref{lem12}.

\smallskip

{\sc Case III}: $Q_l \neq 0$, $d = 4$.  By \eqref{202}, $w$ satisfies the equation
\begin{equation} \label{00318}
\Box_{t,x} w = 3Q_l^2 w + 3 Q_l w^2 + w^3.
\end{equation}
Define $v$ as above, then
\begin{align} \label{transeqd4}
\Box_{s,y} v &= 3 I^{-2} \overline{Q_l}^2  v + 3 I^{-\frac12} \overline{Q_l}  v^2 + I v^3  \nonumber \\
&=: \sum_{j = 1}^3 G_j^{(4)} v^j.
\end{align}
As before, we extend $G_j^{(4)}$ by zero into $|y|<s$, and view $v$ as a solution to the corresponding Cauchy problem, and \eqref{3009}--\eqref{v-reg-ext2} with $u$ replaced by $w$ are still valid. Similar to the derivation of \eqref{G3bound} and \eqref{dG3bound}, here we obtain, in $|y| > s \ge 0$,
\begin{align} \label{G4bound}
\left|G_j^{(4)}(s,y)\right| &\lesssim K(Q)^{3-j} I^{-\frac12 j + \frac52} |y|^{2j - 6} \le K(Q)^{3-j} |y|^{j-1}, \quad 1 \le j \le 3,
\end{align}
and
\begin{align} \label{dG4bound}
\left|\nabla_y G_j^{(4)}(s,y)\right| &\lesssim K(Q)^{3-j} |y|^{j-2}, \quad 1 \le j \le 3.
\end{align}
Similarly, \eqref{G4bound} shows that $G_j^{(4)}$ is continuous at $|y| = s > 0$ after extending by $0$ into $|y| < s$.

\section{Regularity and vanishing of $v$ on the light cone} \label{sec4}

 A crucial preparation for the application of unique continuation later is the nontrivial regularity of $v$ on the forward light cone. According to Section \ref{sec2}, at least we will need $\Box v \in L^\frac{2(d+1)}{d+3}(\Omega)$, $v \in L^{\frac{2(d+1)}{d-1}}(\Omega)$, and  $\partial_{s,y} v \in L^\frac{2(d+1)}{d+3}(\Omega)$ where $\Omega$ is the domain (containing certain parts of the forward light cone) to carry out unique continuation. The last condition on $\partial_{s,y} v$ is equivalent to certain weighted estimates on $\partial_{t,x} u$ at null infinity which do not follow from standard Strichartz estimates. In this section, the maximal regularity of $v$ will be obtained using the decay of $\dot{H}^1 \times L^2$ energy for $u$ in the (future) null channel and the transformed equations \eqref{transeq}, \eqref{transeqd3} and \eqref{transeqd4}.  Then, we prove a key vanishing result for $v$ in certain spacetime domains.
 
\begin{lemma} \label{lem12}
Under either of the assumptions from Theorem \ref{thm2} on the initial data $(u_0, u_1)$, suppose for some $A > 0$,
\begin{equation} \label{401}
\lim_{t \to + \infty} \int_{|x|> A+ t} \left( |\partial_t u|^2 + |\nabla u|^2 \right)  \mathrm{d}x = 0,
\end{equation}
then the transformed solution $v$ (defined in section \ref{sec3}) satisfies 
\begin{equation}
\|v\|_{L_{s,y}^{\frac{2(d+1)}{d-2}}(\mathcal{B})} < \infty,
\end{equation}
where $\mathcal{B} = \{ |y| > s > 0 , \ 0 < |y| + s < A^{-1}\}$.
\end{lemma}

\begin{proof}
By assumption, for a sufficiently large time $t = M_0 > 0$,
\begin{equation} \label{40100}
 \int_{|x|> A + t, t=M_0} \left( |\partial_t u|^2 + |\nabla u|^2 \right) \mathrm{d}x \le \delta_0,
\end{equation}
where $\delta_0$ is a small number to be chosen later. Consider the time-shifted and rescaled solutions
\begin{align*}
u^{(M_0)}(t,x) &= M_0^{\frac{d}{2}-1} u(M_0t+M_0, M_0 x), \\
Q_l^{(M_0)}(t,x) &= M_0^{\frac{d}{2}-1} Q_l(M_0t+M_0, M_0 x).
\end{align*}
Here, $Q_l$ is set as $0$ if $u$ satisfies the assumption (i) from Theorem \ref{thm2}. Then $w^{(M_0)} = u^{(M_0)} - Q_l^{(M_0)}$ has initial data compactly supported in $\overline{B_{M_0^{-1}(R+M_0)}}$. By the decay of $Q$, and choosing $M_0$ large, we can ensure that
\begin{equation} \label{wedecay}
 \int_{|x|> A + t, t=M_0} \left( |\partial_t w|^2 + |\nabla w|^2 \right) \mathrm{d}x \le 2\delta_0.
\end{equation}
Applying the transform in Section \ref{sec3}  to $w^{(M_0)}(t,x)$, we construct a solution $v^{(M_0)}$ to \eqref{transeq}, \eqref{transeqd3} or \eqref{transeqd4} with initial data given by 
\begin{equation}
v^{(M_0)}_0(y) = |y|^{-d+1} {w^{(M_0)}}(0,\frac{y}{|y|^2}), \quad v^{(M_0)}_1 = |y|^{-d-1} \partial_t {w^{(M_0)}}(0,\frac{y}{|y|^2}).
\end{equation}
Note that $v^{(M_0)}_0 = v^{(M_0)}_1 = 0$ for $|y| < M_0 (R+M_0)^{-1}$. Direct computation shows that
\begin{align} \label{vesmall}
&\quad \ \int_{|y| < \frac{M_0}{A+M_0}} \left( \left|v^{(M_0)}_1\right|^2 + \left|\nabla_y v^{(M_0)}_0\right|^2 \right) \mathrm{d}y \nonumber \\ 
&\lesssim \left(\frac{M_0}{R+M_0}\right)^{-2} \int_{|x|>\frac{A+M_0}{M_0}} \left( \left|\partial_t w^{(M_0)}(0,x)\right|^2 + \left|\nabla_x w^{(M_0)}(0, x)\right|^2 \right) \mathrm{d}x  \nonumber \\
&= \left(\frac{M_0}{R+M_0}\right)^{-2} \int_{|x|>A+M_0} \left( \left|\partial_t w(M_0,x)\right|^2 + \left|\nabla_x w(M_0, x)\right|^2 \right) \mathrm{d}x  \nonumber \\
&\le 2 \left(\frac{M_0}{R+M_0}\right)^{-2} \delta_0.
\end{align}
By choosing $\delta_0$ and $\frac{1}{M_0}$ sufficiently small, the last line can be arbitrarily small,

\smallskip
{\sc Case I}: $Q_l = 0$. Recall that $v^{(M_0)}$ solves the equation \eqref{transeq}. Applying \eqref{vesmall} and the small-data well-posedness result in Lemma \ref{wp} to \eqref{transeq}, we get $v^{(M_0)}$ is regular in the backward cone $\mathcal{C} = \{ |y| + |s| < \frac{M_0}{A+M_0}, s \ge 0 \}$. In particular, we have the energy and Strichartz estimates,
\begin{equation} \label{vM0reg}
\partial_{s,y} v^{(M_0)}  \in {L^\infty_s L^2_y(\mathcal{C})}, \quad v^{(M_0)} \in L_{s,y}^{2(d+1)/(d-2)}(\mathcal{C}).
\end{equation}

\smallskip

{\sc Case II}: $Q_l \neq 0$, $d=3$. $v^{(M_0)}$ solves the equation \eqref{transeqd3} with $Q_l$ replaced by $Q_l^{(M_0)}$. Without loss of generality, consider $l = (l_1, 0, 0, 0)$. Note that
\begin{equation}
Q_l^{(M_0)}(t,x) = Q^{(M_0)}\left(\frac{x_1 - l_1 t}{\sqrt{1-l_1^2}}, x_2, x_3\right),
\end{equation}
where $Q^{(M_0)}$ is a new steady-state,
\begin{equation}
Q^{(M_0)}(x_1, x_2, x_3) = M_0^{\frac{d}{2}-1} Q\left( M_0 \left(x_1 - \frac{l_1}{\sqrt{1-l_1^2}}\right), M_0 x_2, M_0 x_3 \right).
\end{equation}
We need an estimate similar to \eqref{G3bound} but with $K(Q)$ modified. Let 
$$\tilde{l} = \left(\frac{l_1}{\sqrt{1-l_1^2}}, 0, 0 \right),$$ 
then we have, for $|x|>t+1$,
\begin{equation}
\left|\left(\frac{x_1 - l_1 t}{\sqrt{1-l_1^2}}, x_2, x_3\right) - \tilde{l}\right| \ge c_0 > 0,
\end{equation}
with $c_0$ depending only on $l_1$. Define
\begin{equation}
\tilde{K}(Q) = \sup_{x \in \mathbb{R}^d, |x-\tilde{l}| \ge c_0} \left\{ |x|^{d-2}|Q(x)| + |x|^{d-1}|\nabla Q(x)| \right\} < + \infty.
\end{equation}
For $(s, y) \in \mathcal{C}$, the corresponding $(t,x)$ satisfies $|x|>t+1$. Repeating the proof of \eqref{G3bound} and \eqref{dG3bound}, we see that they still hold with $(s,y) \in \mathcal{C}$ and with $K(Q)$ replaced by $\tilde{K}(Q)$. Using the decay of $Q$, one can check that,
\begin{equation}
\tilde{K}(Q^{(M_0)}) \to 0, \quad \mathrm{as} \ M_0 \to \infty.
\end{equation}
Hence, by choosing $M_0$ large, we can apply the small-data well-posedness result to $v^{(M_0)}$ localized in $\mathcal{C}$ and deduce the same bounds in \eqref{vM0reg}.

\smallskip

{\sc Case III}: $Q_l \neq 0$, $d=4$. This case is similar to Case II above: by choosing $M_0$ large, using the small-data well-posedness result for $v^{(M_0)}$ localized in $\mathcal{C}$, the regularity estimates \eqref{vM0reg} are again valid.

\smallskip

Next, we use that, in $\{|y| > s \ge 0\} \cap \{t = I^{-1}s \ge M_0\}$, $v^{(M_0)}$ and $v$ are related by
\begin{equation} \label{vvM0}
v(s, y) = M_0^{1 - \frac{d}{2}} J^{\frac{1-d}{2}} v^{(M_0)}\left(\frac{s - IM_0}{M_0J}, \frac{y}{M_0J}\right),
\end{equation}
where $J = -|y|^2 + (s+M_0^{-1})^2$. For $|y| = s > 0$, $J > 0$. The mapping
\begin{equation}
(s, y) \mapsto (s', y') = \left(\frac{s - IM_0}{M_0J}, \frac{y}{M_0J}\right)
\end{equation}
is non-singular at the forward light cone $|y| = s >0$ and sends $\{ (2R)^{-1} < |y| = s <(2A)^{-1}\}$ to $\{ \frac{M_0}{2(R+M_0)} < |y'| = s' < \frac{M_0}{2(A+M_0)}\}$. Combining \eqref{vM0reg}, \eqref{vvM0} and \eqref{3011}, we deduce 
\begin{equation}
v \in L_{s,y}^{\frac{2(d+1)}{d-2}}(\mathcal{B}).
\end{equation}
\end{proof}

\begin{lemma} \label{lem-trace}
Under the assumptions of Lemma \ref{lem12}, we have, as $M \to +\infty$,
$$\int_{\Sigma_M} \left|v\right|^2 \mathrm{d}s\mathrm{d}S_y \lesssim_{A,R} \int_{\Sigma_M} \left|(\partial_s + \partial_{|y|})v\right|^2 \mathrm{d}s\mathrm{d}S_y \to 0.$$
where $ \Sigma_M =  \{ |y| -s = M^{-1}, \, R^{-1} < |y| + s < A^{-1}, \, s > 0\} $. 
\end{lemma}
\begin{proof}
Recall that $w = u - Q_l$ and the case $Q_l = 0$ is included. Multiplying \eqref{202} by $\partial_t u$ we have
\begin{align} \label{0402}
\partial_t \left( \frac12 |\partial_t u|^2 + \frac12 |\nabla u|^2 \right) - \div (\partial_t u \nabla u) = \partial_t \left( \frac{d-2}{2d} |u|^{\frac{2d}{d-2}}\right).
\end{align}
For any $M > R , A \le B < R$, integrate \eqref{0402} in the spacetime domain 
$$U_{M,B} = \{|x|+t \ge 2M + B, B \le |x|-t \le R, t \le M\}.$$
Since $\partial U_{M,B}$ consists of three parts, 
\begin{align*}
\Sigma_{1,M,B} &= \{ M + \frac{B-R}{2} \le t \le M, \, |x| + t = 2 M + B\}, \\
\Sigma_{2,M,B} &= \{t=M, \, M + B <|x| < R + M\}, \\
\Sigma_{3,M,B} &= \{ M + \frac{B-R}{2} \le t \le M, \, |x| - t = R\}.
\end{align*}
we deduce the energy estimate,
\begin{align} \label{ee}
&\quad \int_{\Sigma_{1,M,B}} \left(\frac12 \left|(\partial_t - \partial_r) u\right|^2 + \frac12 \left|\frac{\partial_\omega u}{r}\right|^2 - \frac{d-2}{2d} |u|^{\frac{2d}{d-2}} \right) \mathrm{d}t\mathrm{d}S  \nonumber \\
&=  \int_{\Sigma_{2,M,B}} \left(\frac12 |\partial_t u|^2 + \frac12 |\nabla u|^2 - \frac{d-2}{2d} |u|^{\frac{2d}{d-2}} \right) \mathrm{d}x \nonumber \\
&\quad - \int_{\Sigma_{3,M,B}} \left(\frac12 \left|(\partial_t + \partial_r) u\right|^2 + \frac12 \left|\frac{\partial_\omega u}{r}\right|^2 - \frac{d-2}{2d} |u|^{\frac{2d}{d-2}} \right) \mathrm{d}t\mathrm{d}S \nonumber \\
&=: I_2 + I_3.
\end{align}
Here, $dS$ is the volume element on spheres. By the compact support assumptions, on $\Sigma_3$, $u = Q_l$ and $w = 0$. By \eqref{Qdecay} and \eqref{Qldefinition}, in $d = 3, 4$ all three integrals in \eqref{ee} with $u$ replaced by $Q_l$ decay to zero as $M \to +\infty$. In $d = 5$, we only consider $Q_l = 0$. Hence, with $\varepsilon(M)$ representing a sequence of numbers converging to $0$ (uniformly in B) as $M \to \infty$, we have
$I_3 = \varepsilon(M)$. By \eqref{401} and Sobolev embedding, we also have $I_2 = \varepsilon(M)$ as $M \to \infty$. Hence,
\begin{align} \label{wee}
 \int_{\Sigma_{1,M,B}} \left(\frac12 \left|(\partial_t - \partial_r) u\right|^2 + \frac12 \left|\frac{\partial_\omega u}{r}\right|^2 - \frac{d-2}{2d} |u|^{\frac{2d}{d-2}} \right) \mathrm{d}t\mathrm{d}S  = \varepsilon(M).
\end{align}
It is shown in \cite[Lemma 3.2]{DJKM} that the trace of $|u|^{\frac{2d}{d-2}}$ on the light cone $|y| = s$ is well-defined. Hence the integral
\begin{align}
T_{M,B} := \int_{\Sigma_{1,M,B}} \left(\frac12 \left|(\partial_t - \partial_r) w\right|^2 + \frac12 \left|\frac{\partial_\omega w}{r}\right|^2 \right) \mathrm{d}t\mathrm{d}S  
\end{align}
is finite and continuous in $B$. By definition, $T_{M,B} = 0$ for $B = R$. By Sobolev embedding on $\Sigma_{1,M,B}$ and the compact support of $w$, we have
\begin{equation}
\int_{\Sigma_{1,M,B}} \frac{d-2}{2d} |w|^{\frac{2d}{d-2}} \mathrm{d}t\mathrm{d}S \lesssim T_{M,B}^{\frac{d}{d-2}}.
\end{equation}
Hence, using \eqref{wee}, the decay of $Q_l$ and a continuous induction argument in $B$, we get that $T_{M,B} \le \varepsilon(M)$ for all $B \in [A,R]$. In particular, we have
\begin{equation} \label{4018}
\int_{\Sigma_{1,M,A}} |(\partial_t - \partial_r)w|^2 \mathrm{d}t\mathrm{d}S   = \varepsilon(M).
\end{equation}
By the definition \eqref{3002} of $v$ and \eqref{4018},
$$\int_{\Sigma'_{1,M,A}} |(\partial_s + \partial_{|y|})v|^2 \mathrm{d}s\mathrm{d}S_y = \varepsilon(M).$$
where $ \Sigma'_{1,M,A} =  \{ |y| -s = (2M + A)^{-1}, \, R^{-1} < |y| + s < A^{-1}, \, s > 0\} $. Redefining $M$, and using Poincar\'e's inequality we get the conclusion.

\end{proof}

\begin{corollary} \label{cor12}
Under the assumptions of Lemma \ref{lem12}, $v$ is locally regular up to $s = (2A)^{-1}$, and vanishes in $\{(s,y): |y| \le s < (2A)^{-1}\}$.
\end{corollary}
\begin{proof}

Take a large truncated backward light cone $\mathcal{K} = \{|y| < T-s, 0 \le s \le T_0\}$ with $T > 4 A^{-1}$,  $T_0 = T/2$ and we apply Lemma \ref{wp} in $\mathcal{K}$ to the equation of $v$, that is, \eqref{transeq}, or \eqref{transeqd3}, or \eqref{transeqd4}.  By Remark \ref{rmk-timeexistence}, $v$ is locally regular up to $s=(2R)^{-1} + \delta$ for some $\delta > 0$.

Suppose that the solution $v$ in $\mathcal{K}$ blows up at time $(2R)^{-1} < T^+ < (2A)^{-1}$, let us derive a contradiction. For any $T' < T^+$, we know that
\begin{equation}
\|(v, \partial_s v)\|_{C_s(H^1\times L^2)(\mathcal{K} \cap \{s \le T'\})} + \|v\|_{L_s^\frac{d+2}{d-2}L_y^{\frac{2(d+2)}{d-2}}(\mathcal{K} \cap \{s \le T'\})} < +\infty.
\end{equation}
Using \cite[Lemma 3.2]{DJKM}, the trace of $|v|^{\frac{2d}{d-2}}$ is well-defined on the part of forward light cone $\mathcal{L} = \{|y| = s\}   \cap \{s < T^+\}$. Lemma \ref{lem-trace} and \eqref{3012} shows that $v = 0$ on $\mathcal{L}$. Now notice that $\Box v = 0$ in the region $\{(s,y): |y| < s < T^+\}$, and the energy flux across $\mathcal{L}$ is zero. Using \eqref{3012} and the energy identity\footnote{To justify this identity for $v$, one can use smooth approximations as in the proof of \cite[Lemma 3.2]{DJKM}. Note that the local solutions in Lemma \ref{wp} depend continuously on the initial data and on the coefficients.} for any $(2R)^{-1} < T' < T^+$,
\begin{align} \label{lee}
& \int_{\Sigma_1} \left(\frac12 \left|\partial_t  v\right|^2 + \frac12 |\nabla v|^2  \right) \mathrm{d}x - \int_{\Sigma_2} \left(\frac12 \left|\partial_t  v\right|^2 + \frac12 |\nabla v|^2  \right) \mathrm{d}x \nonumber \\
&\quad = \int_{\Sigma_{3}} \left(\frac12 \left|(\partial_t + \partial_{|y|}) v\right|^2 + \frac12 \left|\frac{\partial_\omega v}{r}\right|^2 \right) \mathrm{d}t\mathrm{d}S,
\end{align}
where
\begin{align*}
\Sigma_1 &= \{s = T', |y| < T'\},\\
\Sigma_2 &= \{s = (2R)^{-1}, |y| < (2R)^{-1}\}, \\
\Sigma_3 &= \{ (2R)^{-1} < s = |y| < T'\},
\end{align*}
we deduce that $v = 0$ for $|y| < s < T^+$. Hence, using Lemma \ref{lem12} and \eqref{3011} we have
\begin{equation}
v \in L_{s,y}^{\frac{2(d+1)}{d-2}}\left(\mathcal{K} \cap \{s < T^+\}\right). 
\end{equation}
This is a contradiction with the blow-up criterion in Lemma \ref{wp}. The conclusion of the lemma now follows.

\end{proof}

\section{Proof of the main results} \label{sec5}

\begin{proof} [Proof of Theorem \ref{thm2}]
Recall that $w = u - Q_l$ and $v$ is the extended transformed solution introduced in Section \ref{sec3}. By assumption (i) (where $Q_l$ is set as zero) or (ii), $w$ has compact spatial support.

\smallskip

\textbf{Statement (A)}: The proof is divided into three steps. 

\smallskip 

\textbf{Step 1}: we prove the vanishing of $w$ outside a timelike hyperboloid assuming that $\mathrm{supp}\,(w_0, w_1) \subset \overline{B_R}$ and
\begin{align} \label{step1ass}
 \lim_{t \to \pm \infty} \int_{|x| > |t|} \left( |\partial_t u|^2 + |\nabla u|^2 \right) \mathrm{d}x = 0.
\end{align}
Using \eqref{step1ass} and Corollary \ref{cor12} with arbitrary $A>0$, we know that $v$ is global for positive times, and vanishes inside the forward light cone $\{(s,y): s > |y|\}$. On the other hand, the compact support of $w$ implies that $v$ vanishes in the backward cones $\{(s,y): |s| + |y| \le R^{-1}\}$. As a result, we are in a similar geometric setting as those considered in \cite[Theorem 1.1]{IK} and \cite{Ler}. For our purpose, we have to find the maximal range of unique continuation. The regularity for $v$ proved in Corollary \ref{cor12} ensures that, for any compact $\mathcal{K} \subset \{(s,y): s\ge 0\}$,
\begin{equation} \label{v-reg}
|\nabla_{s,y} v | \in {L^\infty_s L^2_y(\mathcal{K})}, \quad v \in L_{s,y}^{\frac{2(d+1)}{d-2}}(\mathcal{K}).
\end{equation}

\smallskip

{\sc Case I}: $Q_l = 0$. The equation \eqref{transeq} for $v$ can be written as
\begin{equation}
(\Box + V_1) v = 0, \quad V_1 = (|y|^2 - s^2)_+^\frac{2}{d-2} |v|^{\frac{4}{d-2}}.
\end{equation}
By \eqref{v-reg}, we have
\begin{equation}
V_1 \in L_{s,y}^{\frac{d+1}{2}}(\mathcal{K}).
\end{equation}

\smallskip

{\sc Case II}: $Q_l \neq 0$, $d = 3$. The equation \eqref{transeqd3} for $v$ can be written as
\begin{equation}
(\Box + V_2) v = 0, \quad V_2 = \sum_{j = 1}^5 G_j^{(3)} v^{j-1}.
\end{equation}
By \eqref{v-reg} and \eqref{G3bound}, we also have
\begin{equation}
V_2 \in L_{s,y}^{\frac{d+1}{2}}(\mathcal{K}).
\end{equation}

\smallskip

{\sc Case III}: $Q_l \neq 0$, $d = 4$. The equation \eqref{transeqd4} for $v$ can be written as
\begin{equation}
(\Box + V_3) v = 0, \quad V_3 = \sum_{j = 1}^3 G_j^{(4)} v^{j-1}.
\end{equation}
By \eqref{v-reg} and \eqref{G4bound}, we also have
\begin{equation}
V_3 \in L_{s,y}^{\frac{d+1}{2}}(\mathcal{K}).
\end{equation}

\smallskip

Hence, in each case, we are ready to apply Lemma \ref{lem8}. To find the optimal unique continuation range, we define the strongly pseudoconvex foliation 
\begin{equation}
\Gamma_{a,\delta} : = \left\{- (|y|-\delta)^2 + (s- 2\delta)^2 + a^2 = 0, |y| > \delta, s > 0\right\},
\end{equation}
with $0 < a^2 < (R^{-1} -\delta)^2 - 4\delta^2$ and  $0 < \delta \ll R^{-1}$. The orientation here is given by 
$$\Gamma_{a,\delta}^+ : = \left\{- (|y|-\delta)^2 + (s- 2\delta)^2 + a^2 > 0, s > 0\right\}.$$
When $a$ is sufficiently small, $v$ vanishes in a neighbourhood of $\Gamma_{a,\delta}$. For each $0 < a^2 < (R^{-1} -\delta)^2 - 4\delta^2$, we know that $\Gamma_{a,\delta} \, \cap \, \mathrm{supp} \, v$ is compact. Lemma \ref{lem8} and a continuous induction argument in $a$ gives $v = 0$ in $\Gamma_{R^{-1} - \delta, \delta}^+$. The arbitrariness of $\delta$ gives that
\begin{equation} \label{vvanish1}
v = 0, \quad \mathrm{for} \ |y|^2 - s^2 < R^{-2}, s \ge 0. 
\end{equation}
By \eqref{vvanish1}, going back to the $(t,x)$ coordinates, we deduce that
\begin{equation} \label{wvanishpara}
w = 0, \quad \mathrm{for} \ |x|^2 - t^2 > R^2, t \ge 0.
\end{equation}
The condition $t \ge 0$ in \eqref{wvanishpara} can be removed by applying the same argument to the time-reversed version of $u$ and using \eqref{step1ass}.

\begin{figure}[!h]
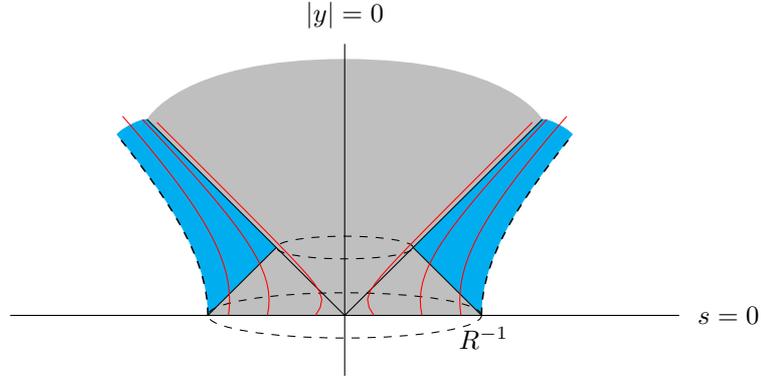

\centering
\tikz[scale = 2] {

\fill [cyan, domain=0:1.2, variable=\x]
  (0.45, 0.45)
  -- plot ( {sqrt(0.81+\x*\x)} , \x)
 -- plot [smooth, tension = 1.5] coordinates {(1.5,1.2) (1.4, 1.27) (1.3, 1.3)}
  -- cycle;

\fill [cyan, domain=0:1.2, variable=\x]
  (-0.45, 0.45)
  -- plot ( {-sqrt(0.81+\x*\x)} , \x)
  -- plot [smooth, tension = 1.5] coordinates {(-1.5,1.2) (-1.4, 1.27) (-1.3, 1.3)}
  -- cycle;

\fill [lightgray]
  (-0.9, 0) -- (0.9, 0) -- (0.45, 0.45) -- (1.3,1.3) -- plot [smooth, tension = 1.5] coordinates {(1.3,1.3) (0,1.7) (-1.3, 1.3)} -- (-1.3 , 1.3) -- (-0.45,0.45) -- cycle;

\draw [dashed] (0,0) ellipse (0.9 and 0.15);
\draw (-2.2,0) -- (2.2,0);
\draw (0, -0.4) -- (0,1.8);
\draw (0, 0) -- (1.3, 1.3);
\draw (0, 0) -- (-1.3, 1.3);

\draw (0.9, 0) -- (0.45 , 0.45);
\draw (-0.9, 0) -- (-0.45, 0.45);
\draw [dashed] (0,0.45) ellipse (0.45 and 0.075);

\draw [line width = 0.2mm, dashed, domain=0:1.2] plot ( {sqrt(0.81+\x*\x)} , \x) ;

\draw [line width = 0.2mm, dashed, domain=0:1.2] plot ( {-sqrt(0.81+\x*\x)} , \x) ;

\draw [color = red, domain=-0.1:1.22] plot ( {sqrt(0.5+\x*\x)+0.05} , \x + 0.1) ;
\draw [color = red, domain=-0.1:1.2] plot ( {sqrt(0.2+\x*\x)+0.05} , \x + 0.1) ;
\draw [color = red, domain=-0.1:1.18] plot ( {sqrt(0.01+\x*\x)+0.05} , \x + 0.1) ;

\draw [color = red, domain=-0.1:1.22] plot ( {-sqrt(0.5+\x*\x)-0.05} , \x + 0.1) ;
\draw [color = red, domain=-0.1:1.2] plot ( {-sqrt(0.2+\x*\x)-0.05} , \x + 0.1) ;
\draw [color = red, domain=-0.1:1.18] plot ( {-sqrt(0.01+\x*\x)-0.05} , \x + 0.1) ;

\node[label=left:{$s = 0$}]  at (2.85, 0) {};
\node[label=right:{}]  at (-2.85, 0) {};
\node[label=below:{$|y| = 0$}]  at (0, 2.2) {};

\node[label=left:{$R^{-1}$}]  at (1.2, -0.15) {};
}

\caption{In $(s, y)$ coordinates. Using the vanishing of $v$ in the gray region, we prove the vanishing of $v$ in the blue region. The red curves represent the strongly pseudoconvex hypersurfaces $\Gamma_{a, \delta}$.}
\end{figure}


\smallskip

\textbf{Step 2}: we prove a weaker energy channel property: it is impossible that 
\begin{align} \label{step2ass}
 \lim_{t \to \pm \infty} \int_{|x|>-\frac{\e}{2} + |t|} \left( |\partial_t u|^2 + |\nabla u|^2 \right) \mathrm{d}x = 0.
\end{align}
Let us argue by contradiction, suppose \eqref{step2ass} holds. Without loss of generality, assume $\mathrm{supp}\,(w_0, w_1) \subset \overline{B_R}$ and there is a point $x_0 \in \mathrm{supp}\,(w_0, w_1)$ with $|x_0| = R$.  By Step 1, for any small $k > 0$, there exists $\nu(k) > 0$ such that
\begin{equation} \label{vanish-line}
w = 0, \quad \mathrm{for} \ |x| > k|t| + R, \  |t| < \nu(k).
\end{equation}
In the above unique continuation argument, we only need \eqref{step2ass} for $\e = 0$. The same method can be applied to the time-shifted solutions $T_\tau w(t, x) := w(t + \tau, x) = u(t + \tau, x) - Q_l(t+\tau,x)$, $|\tau| < \varepsilon$. According to \eqref{vanish-line}, for $\tau < \nu(k)$, the support of $\left(T_\tau w(0, \cdot), \partial_t T_\tau w(0,\cdot) \right)$ is contained in $\overline{B_{k|\tau| + R}}$. The number $\nu(k)$ in \eqref{vanish-line} can be chosen such that \eqref{vanish-line} holds with  $R$ slightly varied. Similar to \eqref{vanish-line} we have
\begin{equation} \label{vanish-line11}
T_\tau w = 0, \quad \mathrm{for} \ |x| > k|t| + R + k |\tau|, \  |t| < \nu(k).
\end{equation}
Induction in $\tau$ gives 
\begin{equation} \label{vanish-line2}
w = 0, \quad \mathrm{for} \ |x| > k|t| + R, \  |t| < \frac{\e}{4}.
\end{equation}
By the arbitrariness of $k$,
\begin{equation} \label{vanish-line3}
w = 0, \quad \mathrm{for} \ |x| > R, \ |t| < \frac{\e}{4}.
\end{equation}
Now, using the Strichartz estimate
\begin{equation}
|u|^{\frac{4}{d-2}} \in L_{t,x,\loc}^{\frac{d+1}{2}},
\end{equation}
and Lemma \ref{lem8} applied to the equation of $w$, and the fact $\Gamma = \{|x| = R\}$ with $\Gamma^+ = \{|x| > R\}$ is strongly pseudoconvex, we deduce that $w$ vanishes near the sphere $\{t=0, |x| = R\}$.  This is a contradiction with our initial assumption that $|x_0| = R$ and $x_0 \in \mathrm{supp}\, (w_0, w_1)$.

\smallskip

\textbf{Step 3}: We prove (A) by contradiction. Suppose both (A1) and (A2) are false. Reversing the time direction if necessary, we have
\begin{align} \label{502}
0 &= \liminf_{t \to + \infty} \int_{|x|> t} \left( |\partial_t u|^2 + |\nabla u|^2 \right) \mathrm{d}x \nonumber \\
&= \liminf_{t \to - \infty} \int_{|x|>-\varepsilon -t} \left( |\partial_t u|^2  + |\nabla u|^2 \right) \mathrm{d}x.
\end{align}
The first line implies that, for any $\delta > 0$, there exists $M_0 > 1$ such that
\begin{equation}
\int_{|x|> M_0} \left( |\partial_t u|^2 + |\nabla u|^2 \right)(M_0, x) \mathrm{d}x \le \delta.
\end{equation}
By Hardy estimate and Sobolev extension, there exists $(\tilde{u}_0, \tilde{u}_1)$ such that
\begin{equation}
(\tilde{u}_0(x), \tilde{u}_1(x)) = (u(M_0, x), u_1(M_0, x)), \quad \mathrm{for} \ |x| >  M_0,
\end{equation}
and
\begin{equation}
\|\tilde{u}_0\|_{\dot{H}^1(\mathbb{R}^d)}^2 + \|\tilde{u}_1\|_{L^2(\mathbb{R}^d)}^2 \lesssim \delta.
\end{equation}
Using the small-data well-posedness result for \eqref{202},  the finite speed of propagation and the arbitrary smallness of $\delta$, we deduce that 
\begin{align} \label{5005}
 \lim_{t \to + \infty} \int_{|x|> t} \left( |\partial_t u|^2 + |\nabla u|^2 \right) \mathrm{d}x = 0.
\end{align}
Similarly, for negative times, we have
\begin{align} \label{5006}
 \lim_{t \to - \infty} \int_{|x|>-\varepsilon -t} \left( |\partial_t u|^2 + |\nabla u|^2 \right) \mathrm{d}x = 0.
\end{align}
Now, consider the time-shifted solutions
\begin{align}
T_{-\e/2}u(t,x) = u(t - \frac{\e}{2}, x), \\
T_{-\e/2}w(t, x) = w(t - \frac{\e}{2}, x).
\end{align}
\eqref{5005}--\eqref{5006} are equivalent to
\begin{align}
 \lim_{t \to \pm \infty} \int_{|x|>-\varepsilon/2 + |t|} \left( |\partial_t (T_{-\e/2}u)|^2 + |\nabla (T_{-\e/2}u)|^2 \right) \mathrm{d}x = 0,
\end{align}
which is a contratiction with Step 2 applied to the time-shifted solutions. Hence, the proof of (A) is finished.

\smallskip

\textbf{Statement (B)}:  Let us argue by contradiction. Suppose the conclusion is false, that is,
\begin{align} \label{5010}
\liminf_{t \to + \infty} \int_{|x|>  - R + t} \left( |\nabla u|^2 + |\partial_t u|^2 \right) = 0.
\end{align}
Using the derivation of \eqref{5005}, \eqref{5010} can be improved to
\begin{align} \label{5011}
\lim_{t \to + \infty} \int_{|x|>  - R + t} \left( |\nabla u|^2 + |\partial_t u|^2 \right) = 0.
\end{align} 
Applying the same method from Step 2 above, we obtain (see the derivation of \eqref{vanish-line3})
\begin{equation}
w = 0, \quad \mathrm{for} \ |x|> R, \  0 \le t \le R,
\end{equation}
and also (see the derivation of \eqref{wvanishpara})
\begin{equation}
w = 0, \quad \mathrm{for} \ |x|^2 - (t - R )^2 > R^2 , \  t > R.
\end{equation}

\begin{figure}[!h]
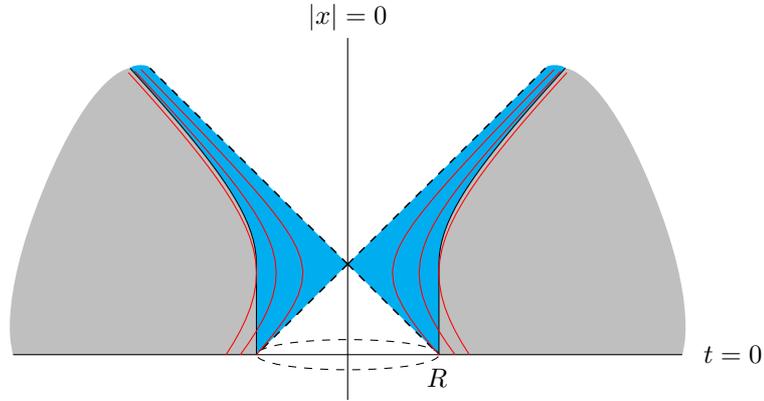

\centering
\tikz[scale = 2] {

\fill [cyan, domain=0:1.2, variable=\x]
  (0.6, 0) -- (0.6, 0.6)
  -- plot [domain=0:1.3] ( {sqrt(0.36+\x*\x)} , \x + 0.6) -- plot [smooth, tension = 1.5] coordinates {(1.4318, 1.9) (1.36,1.92) (1.3,1.9)} 
  -- (0,0.6) 
  -- cycle;

\fill [cyan, domain=0:1.2, variable=\x]
  (-0.6, 0) -- (-0.6, 0.6)
  -- plot [domain=0:1.3] ( -{sqrt(0.36+\x*\x)} , \x + 0.6) -- plot [smooth, tension = 1.5] coordinates {(-1.4318, 1.9) (-1.36,1.92) (-1.3,1.9)} 
  -- (0,0.6) 
  -- cycle;

\fill [lightgray]
  (0.6, 0) -- (0.6, 0.6) -- plot [color = red, domain=0:1.3] ( {sqrt(0.36+\x*\x)} , \x + 0.6) -- plot [smooth, tension = 1.5] coordinates {(1.4318, 1.9) (2,1.1) (2.2,0)} -- cycle;

\fill [lightgray]
  (-0.6, 0) -- (-0.6, 0.6) -- plot [color = red, domain=0:1.3] ( -{sqrt(0.36+\x*\x)} , \x + 0.6) -- plot [smooth, tension = 1.5] coordinates {(-1.4318, 1.9) (-2,1.1) (-2.2,0)} -- cycle;

\draw [dashed] (0,0) ellipse (0.6 and 0.1);

\draw (-2.2,0) -- (2.2,0);
\draw (0, -0.3) -- (0,2.1);
\draw [line width = 0.2mm, dashed] (0, 0.6) -- (1.3, 1.9);
\draw [line width = 0.2mm, dashed] (0, 0.6) -- (-1.3, 1.9);
\draw [line width = 0.2mm, dashed] (0, 0.6) -- (0.6, 0);
\draw [line width = 0.2mm, dashed] (0, 0.6) -- (-0.6, 0);
\draw  plot [ domain=0:1.3] ( {sqrt(0.36+\x*\x)} , \x + 0.6);
\draw  plot [ domain=0:1.3] ( {-sqrt(0.36+\x*\x)} , \x + 0.6);

\draw (0.6, 0) -- (0.6 , 0.6);
\draw (-0.6, 0) -- (-0.6, 0.6);

\draw [red, domain=-0.54:1.33] plot  ( {sqrt(0.63*0.63+\x*\x)-0.03} , \x + 0.54);
\draw [red, domain=-0.54:1.34]  plot ( {sqrt(0.25+\x*\x)-0.03} , \x + 0.54);
\draw [red, domain=-0.54:1.35] plot  ( {sqrt(0.63*0.63-0.54*0.54+\x*\x)-0.03} , \x + 0.54);

\draw [red, domain=-0.54:1.33] plot  ( -{sqrt(0.63*0.63+\x*\x)+0.03} , \x + 0.54);
\draw [red, domain=-0.54:1.34]  plot ( -{sqrt(0.25+\x*\x)+0.03} , \x + 0.54);
\draw [red, domain=-0.54:1.35] plot  ( -{sqrt(0.63*0.63-0.54*0.54+\x*\x)+0.03} , \x + 0.54);

\node[label=left:{$t = 0$}]  at (2.85, 0) {};
\node[label=right:{}]  at (-2.85, 0) {};
\node[label=below:{$|x| = 0$}]  at (0, 2.45) {};

\node[label=left:{$R$}]  at (0.78, -0.15) {};
}

\caption{In $(t, x)$ coordinates. Using the vanishing of $v$ in the gray region, the maximal range of unique continuation is the blue region which contains the Cauchy slice at time $t = R$. The red curves represent the strongly pseudoconvex hypersurfaces $\Gamma'_{a, \delta}$.}
\end{figure}

Next, we use a strongly pseudoconvex foliation $\Gamma'_{a,\delta}$ in $(t,x)$ coordinates, with $0 < \delta \ll R$ and $(R+\delta)^2 - (R-2\delta)^2 < a^2 \le (R + \delta)^2$:
\begin{equation}
\Gamma'_{a,\delta} : = \{\, (|x|+\delta)^2 - (t - R + 2\delta)^2 = a^2, t>0\}.
\end{equation}
Note that for $a^2 = (R+\delta)^2$, $w$ vanishes on the spatial exterior side of $\Gamma'_{a,\delta}$. Hence, using  Lemma \ref{lem8} applied to \eqref{202} and $\Gamma'_{a,\delta}$, and the arbitrary smallness of $\delta$, we get
\begin{equation}
w = 0, \quad \mathrm{for} \ |x| > |t - R| \ \mathrm{and} \ t \ge 0.
\end{equation}
In particular,
\begin{equation}
w = \partial_t w = 0, \quad \mathrm{at} \ t = R.
\end{equation}
Since $u$ satisfies the equation \eqref{202}, we deduce that $w=0$ for all spacetime points $(t,x)$, which is a contradiction with the assumption that $u-Q_l$ is nontrivial. Hence, we have finished the proof.

\end{proof} 

\begin{proof}[Proof of Corollary \ref{thm4}]
First, consider the case $T^+ = + \infty$. It is well-known that solutions satisfying the compactness property must be non-radiative, that is, for any $A \in \mathbb{R}$,
\begin{equation}
\int_{|x|> A +|t|} \left( |\partial_t u|^2 + |\nabla u|^2 \right)  dx \to 0, \quad \mathrm{as} \ t \to +\infty.
\end{equation}
Hence, we can directly apply Theorem \ref{thm2} statement (B) to deduce that $ (u_0 - Q_l(0, \cdot), u_1 - \partial_t Q_l(0, \cdot))$ cannot have compact spatial support (unless $u$ is identical to $Q_l$).

Next, we exclude the case $T^+ < + \infty$. Suppose $T^+ < +\infty$, then it is known that, due to the compactness property, $u$ is supported in the backward light cone $\left\{ |x - x^+| < T^+ - t, t < T^+ \right\}$ where $(T^+, x^+)$ is the unique blow-up point at $T^+$ time (see equation (2.13) in \cite{DKMCPAA}). In particular, in this case $(u_0, u_1)$ is compactly supported. If $T^- = -\infty$, then we can use the the argument in the last paragraph  for the time-reversed version of $u$ to obtain a contradiction. 

The case that both $T^-$ and $T^+$ are finite was excluded in \cite[Proposition 2.1]{DKMCPAA} using integral estimates. Here we give a new argument. In this case, by the compactness property of $u$, $u$ is supported in a compact region in spacetime (see equation (2.5) in \cite{DKMCPAA}), more precisely,
\begin{equation}
\mathrm{supp} \, (u, \partial_t u) \subset [T^-, T^+] \times \mathbb{R}^d \cap \left\{ |x-x^+| \le |t - T^+|, |x - x^-| \le |t - T^-| \right\}.
\end{equation}
Take a unit vector $\omega \in \mathbb{R}^d$ which is perpendicular to $x^+ - x^-$. Note that the vertical hyperplane $\{ (t, x) \in \mathbb{R}^{1+d}, x \cdot \omega = a \}$, $a \in \mathbb{R}$, are not pseudoconvex with respect to $\Box$. However, the vertical cylinders  $\{(t, x) \in \mathbb{R}^{1+d}, |x - N \omega| = N-a \}$ with large $N$ are strongly pseudoconvex (with the appropriate orientation), and locally approximate the vertical hyperplanes. Using foliations of such cylinders and Lemma \ref{lem8}, it is easy to deduce that $u$ vanishes in the spacetime slab $[T^-, T^+] \times \mathbb{R}^d$.  

The proof is now finished.
\end{proof}

\section*{Acknowledgement}
We would like to thank H. Jia for some useful discussions.
The authors were in part supported by NSFC (Grant No. 11725102), National Support Program for Young Top-Notch Talents, and Shanghai Science and Technology Program (Project
No. 21JC1400600 and No. 19JC1420101).

\bibliographystyle{plain}
\bibliography{references}

\end{document}